\pdfoutput=1
\documentclass[11pt]{amsart}
\usepackage{amsmath,amssymb,amsthm}
\usepackage[margin=1.2in]{geometry}
\usepackage{booktabs}
\usepackage{adjustbox}
\usepackage{xcolor}
\usepackage[colorlinks=true,linkcolor=blue,citecolor=blue,urlcolor=blue]{hyperref}
\hypersetup{pdfauthor={Bernd Johannes Wuebben},
  pdftitle={Deformations of toric pairs and the smoothing of Batyrev Calabi-Yau threefolds}}

\newtheorem{theorem}{Theorem}[section]
\newtheorem{proposition}[theorem]{Proposition}
\newtheorem{lemma}[theorem]{Lemma}

\theoremstyle{definition}

\newtheorem{question}[theorem]{Question}
\theoremstyle{remark}
\newtheorem{remark}[theorem]{Remark}

\newcommand{\ZZ}{\mathbb{Z}}
\newcommand{\CC}{\mathbb{C}}
\newcommand{\PP}{\mathbb{P}}
\newcommand{\FF}{\mathbb{F}}
\newcommand{\RR}{\mathbb{R}}
\newcommand{\conv}{\operatorname{conv}}
\newcommand{\Cone}{\operatorname{Cone}}
\newcommand{\Spec}{\operatorname{Spec}}

\newcommand{\Sing}{\operatorname{Sing}}

\title[Deformations of toric pairs]{Deformations of toric pairs\\ and the
smoothing of Batyrev Calabi--Yau threefolds}
\author{Bernd Johannes Wuebben}
\date{August 15, 2026}

\begin{document}

\begin{abstract}
Let $\Delta$ be a reflexive four-polytope with primitive edges. We
construct deformations of the toric fourfold defined by its face fan
that smooth the singularities of a generic anticanonical Calabi--Yau
hypersurface over a chosen two-dimensional face. The construction applies
when the face cone is smoothable and the dual edge has an interior
lattice point; its equations are trinomials in Cox coordinates. It gives
a smoothing of the hypersurface when this is the only face whose cone
is singular. If every singular face satisfies these conditions,
simultaneous smoothing follows in the nodal case and, in general, under
irreducibility of the semiuniversal deformation space. A seven-vertex
example with one ordinary double point and no smoothing shows that the
dual-edge condition is sharp. A scan of the Kreuzer--Skarke classification
shows that the facewise criterion
accounts for 3,774 of the 30,241 smoothings in the Batyrev--Kreuzer census;
the other 26,467 require relations among exceptional curve classes from
different faces.
\end{abstract}

\maketitle

\section{Introduction}\label{sec:intro}

Let $\Delta \subset \ZZ^4$ be a reflexive polytope all of whose edges have
lattice length one, and let $X \subset \PP_\Delta$ be a generic
anticanonical hypersurface in the Gorenstein Fano toric fourfold
associated to the face fan of $\Delta$.  By Batyrev \cite{Batyrev94} $X$
is a Calabi--Yau threefold with Gorenstein canonical singularities, and by
\cite[Prop.~3.1]{paper1} its singular locus
is exactly: for each two-dimensional face $F \preceq \Delta$ whose cone
$C(F)$ is singular, $\ell(F^\circ)$ points at which the germ of $X$ is
analytically isomorphic to $C(F)$.

The deformation theory of the germs is completely understood
combinatorially: by Altmann \cite{Altmann97} and
Corti--Filip--Petracci \cite{CFP}, $C(F)$ is \emph{smoothable} if and only
if $F$ admits a Minkowski decomposition into unit segments and standard
triangles (type (S) in the Altmann--Corti--Filip--Petracci trichotomy,
as recalled in \cite[Thm.~2.2]{paper1}).  The companion paper
proved that a single face of the other two types (reduced-rigid or
deformable-but-non-smoothable) forbids every global smoothing of $X$.
This paper treats the remaining case: \emph{every} singular
two-dimensional face is of type (S), so every singular point of $X$
admits a local smoothing, and the question is whether the local
smoothings are realized by a global deformation.  This is the classical
local-to-global problem: for ordinary double points smoothability is
governed by
Friedman's relation among the exceptional curve classes
\cite{Friedman86,RollenskeThomas}, and for the toric hypersurfaces at
hand it was studied, in the all-conifold case, by Batyrev--Kreuzer
\cite{BK10}, who found that of the $198{,}849$ reflexive $4$-polytopes
whose two-dimensional faces are standard triangles and unit
parallelograms (at least one parallelogram),
only $30{,}241$ could be shown smoothable by Namikawa's criterion.
One dimension down, the analogous constructive problem is to smooth a
Gorenstein toric Fano \emph{threefold} from Minkowski data on the facets
of a reflexive $3$-polytope; it was recently solved by
Corti--Hacking--Petracci \cite{CHP24}.  The obstruction side of the same
programme is due to Petracci \cite{Petracci20}, who produces
non-smoothable Gorenstein toric Fano threefolds by a local-to-global
criterion on a reflexive $3$-polytope.  Theorem~\ref{thm:B} below plays
the corresponding role here: it gives a uniform one-node obstruction and
an explicit reflexive $4$-polytope realizing it.  In both of those works the Minkowski data live on
the \emph{facets} of the $3$-polytope and the object smoothed or
obstructed is the Gorenstein toric Fano threefold itself; here the data
live on the
codimension-two faces of a reflexive $4$-polytope and the object smoothed
is the anticanonical hypersurface in the toric fourfold.

Our first result is constructive and identifies a clean sufficient
condition.  Write $\ell(F^\circ)$ for the lattice length of the edge of
$\Delta^\circ$ dual to a two-dimensional face $F$.

\begin{theorem}[= Theorem~\ref{thm:A2}]\label{thm:A}
Let $\Delta$ be reflexive with unit edges and let $F$ be a
two-dimensional face of type \textup{(S)}, with a Minkowski decomposition
$F = Q_0 + \dots + Q_r$ into unit segments and standard triangles, and
suppose $\ell(F^\circ) \ge 2$.  Then there is an explicit $r$-parameter
deformation of the pair $(\PP_\Delta, \partial\PP_\Delta)$, given by
trinomial equations in Cox coordinates, whose induced deformation of the
generic anticanonical $X$ has generic fibre smooth in a neighbourhood of
the $\ell(F^\circ)$ points of type $C(F)$.  In particular, if $F$ is the
\emph{unique} non-smooth two-dimensional face of $\Delta$, then $X$ is
smoothable.
\end{theorem}

The idea is simple to state.  The deformation is Petracci's homogeneous
deformation of the polarized toric pair, applied not to a facet but to
the codimension-two face $F$: the deformation datum lives on the cone
over the pair, with weight an \emph{interior} lattice point $v^\ast$ of
the dual edge $F^\circ$.  At the two endpoints of $F^\circ$ the relevant
slice of the cone grows to a full facet of $\Delta$, and the datum would
require a Minkowski decomposition of that three-dimensional facet; the
interior points of $F^\circ$ are exactly what makes a face-local datum
possible.  This is the source of the threshold $\ell(F^\circ)\ge2$, and
\S\ref{sec:node} shows the threshold is sharp.

Theorem~\ref{thm:A} belongs to a known genre.  A Calabi--Yau
hypersurface in a toric variety has \emph{polynomial} deformations,
obtained by moving $X$ inside a fixed $\PP_\Delta$, and these never
change the singularity structure here (Remark~\ref{rem:nonpoly}); the
\emph{non-polynomial} deformations, induced instead from deformations of
the ambient toric variety, were constructed by Mavlyutov \cite{Mav04}.
The construction we use is Petracci's \cite{Petracci18}, which extends
this Altmann--Mavlyutov theory of homogeneous deformations
\cite{Altmann95,Mav04} from affine toric varieties to toric pairs, and
Theorem~\ref{thm:A} is of exactly this non-polynomial kind.  What is new
is not the passage
from ambient to hypersurface but the \emph{source} of the ambient
deformation: a datum supported on a single codimension-two face, whose
existence is governed by the lattice length of the dual edge and whose
effect on the germ is read off from the Minkowski decomposition of that
face.  This is what makes the criterion face-local, and hence testable
across the Kreuzer--Skarke list.

For several type-(S) faces the construction smooths the points over each
face individually; the face-wise smoothings assemble unconditionally in
the nodal case, and in general under a hypothesis on the deformation
space.
The construction does not supply compatible toric data on the deformed
fibres with which to repeat the argument; see
Remark~\ref{rem:noiterate}.

\begin{theorem}[= Theorem~\ref{thm:A3}]\label{thm:Amulti}
Let $\Delta$ be reflexive with unit edges and suppose that every
non-smooth two-dimensional face $F_1,\dots,F_s$ of $\Delta$ is of type
\textup{(S)} (a Minkowski sum of unit segments and standard
triangles), with $\ell(F_i^\circ) \ge 2$.  Then:
\begin{enumerate}
\item for every singular point $p$ of $X$ there is an explicit
deformation of $X$ whose generic fibre is smooth in a neighbourhood of
$p$; consequently the locus $D_p \subset (B,0)$ of points of the
semiuniversal deformation of $X$ whose fibres remain singular near $p$
is a proper closed analytic subgerm;
\item if the base germ $(B,0)$ of the semiuniversal deformation of $X$
is irreducible (for instance smooth), then $X$ is smoothable;
\item if every $F_i$ is a unit parallelogram, then $X$ is smoothable
unconditionally.
\end{enumerate}
\end{theorem}

Parts (ii) and (iii) are proved by different mechanisms.  Part (ii) is
soft.  By (i) each singular point is removed
by \emph{some} deformation, so the locus in the semiuniversal base over
which that point stays singular is a proper closed analytic subgerm; the
singular points are finite in number, so there are finitely many such
loci.  A finite union of proper closed analytic subgerms cannot exhaust
an irreducible germ, and an arc through the origin missing all of them
is a smoothing.  Irreducibility is thus precisely what converts
``removable one point at a time'' into ``removable simultaneously'', and
it is a hypothesis we cannot presently discharge: by Gross
\cite{Gross97} the deformation space of a Calabi--Yau threefold with
canonical singularities need not be smooth, and we do not know whether
the base is irreducible for the hypersurfaces at hand
(Question~\ref{q:multi}).  Part (iii) avoids the
hypothesis altogether by working in homology instead of in the
deformation space.  For nodes, the $\ell(F_i^\circ)$ exceptional curves
over one face all carry the \emph{same} class, so Friedman's relation
can be satisfied one face at a time: one chooses non-zero integer
coefficients summing to zero, which is possible exactly when
$\ell(F_i^\circ) \ge 2$, and the per-face relations simply add.  This
is the same threshold as Theorem~\ref{thm:A}, arrived at from homology
rather than from the existence of a deformation datum, and \S\ref{sec:bk}
takes up the coincidence.

The type-(S) hypothesis enters only at the last step, to make the local
general fibre smooth.  Applied to an \emph{arbitrary} Minkowski
decomposition the same face family therefore realizes globally the
cascade of \cite[Lemma~2.4]{paper1}, replacing each of the
$\ell(F^\circ)$ points over $F$ by one point of type $C(Q_i)$ for every
two-dimensional non-smooth summand $Q_i$
(Proposition~\ref{prop:cascade}).  For the deformable-but-non-smoothable
pentagon of \cite{paper1} this yields an explicit trinomial deformation
between two \emph{non-smoothable} Calabi--Yau threefolds, each of
three deformable-but-non-smoothable germs being traded for a rigid
$\tfrac13(1,1,1)$ point.  This is an obstructed analogue of a conifold
transition, in which the deformation genuinely moves but no fibre is ever
smooth (Remark~\ref{rem:cascade-global}).

The hypothesis $\ell(F^\circ)\ge 2$ is not an artifact of the method.

\begin{theorem}[= Theorem~\ref{thm:B2}]\label{thm:B}
Let $\Delta$ be a reflexive $4$-polytope with unit edges.  If the singular
locus of its generic anticanonical hypersurface $X\subset\PP_\Delta$
consists of a single ordinary double point, then $X$ admits no smoothing.

In particular, let
\[
\Delta_N=\conv\bigl(\square\times\{(1,0)\}\ \cup\ \{e_4,\ (0,1,-1,0),\
(-1,-2,1,-1)\}\bigr)\subset\ZZ^4,
\]
where $\square$ is the unit square embedded with vertices
$(0,0),(0,-1),(1,-1),(1,0)$.  Then $\Delta_N$ is reflexive with $7$
vertices and unit edges; its unique non-smooth two-dimensional face is
$\square\times\{(1,0)\}$, of type \textup{(S)}, with dual edge of lattice
length $1$; the generic anticanonical $X_N$ is a Calabi--Yau threefold
with exactly one ordinary double point; and $X_N$ admits no smoothing.
The maximal projective crepant partial (MPCP) resolution $\widehat X_N$ has
$(h^{1,1}, h^{2,1}) = (3, 103)$ and Euler number $-200$.
\end{theorem}

Thus at $\ell(F^\circ)=1$ the local-to-global obstruction is real even
though the local
model is the simplest smoothable singularity there is.  The obstruction
is Friedman's: the maximal projective crepant partial resolution of $X_N$
is a \emph{projective small} resolution, so the exceptional curve class
cannot vanish, while a one-node Calabi--Yau threefold is smoothable only
if it does.  For the following divisorial germs, the later mixed
criterion of \cite{paper5} also gives a nonsmoothability result:

\begin{proposition}[= Proposition~\ref{prop:C2}]\label{prop:C}
Let $\Delta$ be reflexive with unit edges whose unique non-smooth
two-dimensional face $F$ has exactly one interior lattice point, is of
type \textup{(S)}, and has $\ell(F^\circ) = 1$.  Then the singular locus
of $X$ is one point, the anticanonical cone over $S \in
\{\PP^1{\times}\PP^1,\ \mathrm{dP}_7,\ \mathrm{dP}_6\}$, and an MPCP
resolution $\widehat X \to X$ can be chosen whose exceptional locus is
the smooth irreducible surface $S$, contracted to the point.  The
contraction has relative Picard number
$\rho(\widehat X/X) = \rho(S) \in \{2,3,4\}$, so it factors in the projective category and
is \emph{never} a primitive contraction, and $X$ is not
$\mathbb{Q}$-factorial at the point.  The map
$H_2(S;\mathbb Q)\to H_2(\widehat X;\mathbb Q)$ is injective;
the necessary condition of \cite[Thm.~A]{paper5}, recalled in
\S\ref{sec:germ}, therefore implies that $X$ is not smoothable.
\end{proposition}

The statements assemble into a germ-dependent picture of the
local-to-global problem
(Table~\ref{tab:picture}): for type-(S) faces the dual-edge length
$\ell(F^\circ) \ge 2$ smooths the points over the face, and the
face-wise smoothings assemble in the
nodal case and, in general, whenever the deformation space of $X$ is
irreducible (Theorem~\ref{thm:Amulti}); at $\ell(F^\circ) = 1$ the
picture is governed by the crepant resolution over the point.  For the
node the resolution is small and Friedman's relation obstructs
(Theorem~\ref{thm:B}).  For the germs over a face with a single interior
lattice point ($i(F)=1$) the resolution is
divisorial, but, in contrast to the primitive type~II contractions
analyzed by Gross, Proposition~\ref{prop:C} shows the toric
contraction always has relative Picard number $\ge 2$ and $X$ is not
$\mathbb{Q}$-factorial.  The injection of the exceptional surface's
second homology into that of the resolution rules out a smoothing by
the necessary condition recalled in \S\ref{sec:germ}.  At the level of
germs, Gross's
theorem for primitive realizations fails exactly when the exceptional
surface is $\PP^2$ or $\FF_1$, and these are precisely the two classes with one
interior point that are \emph{reduced-rigid} in the trichotomy
(Theorem~\ref{thm:trich}).
The local trichotomy and the global smoothing problem see the same
dichotomy.

\begin{table}[ht]
\adjustbox{max width=\textwidth}{%
\begin{tabular}{lcc}
\toprule
germ of the type-(S) face & $\ell(F^\circ)=1$ & $\ell(F^\circ)\ge 2$\\
\midrule
unit square (node; crepant res.\ small) &
\textbf{non-smoothable} (Thm~\ref{thm:B}) &
smoothable (Thm~\ref{thm:A})\\
$i(F)=1$ (cone over $\PP^1{\times}\PP^1,\mathrm{dP}_7,\mathrm{dP}_6$; divisorial) &
non-smoothable (Prop.~\ref{prop:C}; \cite{paper5}) & smoothable (Thm~\ref{thm:A})\\
$i(F)\ge2$ type (S) (divisorial, reducible exceptional locus) &
open & smoothable (Thm~\ref{thm:A})\\
\bottomrule
\end{tabular}}
\medskip
\caption{The local-to-global picture by germ and dual-edge length, for a
single
singular face.  For several singular faces see
Theorem~\ref{thm:Amulti}: the points over each face are smoothed
individually, the
all-conifold case is unconditional, and the general case assembles when
the deformation space is irreducible; mixed configurations involving
$\ell=1$ nodes are governed by cross-face Friedman relations
(\S\ref{sec:bk}).}\label{tab:picture}
\end{table}

Section~\ref{sec:bk} confronts this picture with the Batyrev--Kreuzer
census.  For the $\ell(F^\circ)$ nodes over a single square face the
exceptional classes coincide, so Friedman's all-nonzero relation is
satisfiable face-locally if and only if $\ell(F^\circ) \ge 2$; this
reproduces the threshold of Theorem~\ref{thm:A} from pure homology.  A
scan of the verified database copy described in \S\ref{sec:bk}
shows that this per-face condition
accounts for exactly $3{,}774$ of the $30{,}241$
conifold polytopes that Batyrev--Kreuzer show to be smoothable, one
in eight, so the great majority of their smoothings require
relations among exceptional curve classes from different faces.

All computations are performed by the exact-arithmetic programs of the
companion paper, extended by a routine enumerating the conifold polytopes of the
Kreuzer--Skarke list; they are
included as ancillary files and are public, with all data, at
\begin{center}
\url{https://github.com/bwuebben/calabi-yau-smoothability}\,.
\end{center}

\subsection*{Acknowledgements}
I am grateful to Mark Gross for first drawing my attention, many years
ago, to the problem of smoothing canonical singularities on Calabi--Yau
threefolds.
The author used LLMs during the development of this work for literature exploration, editorial and coding assistance and as an adversarial reader.

\section{Preliminaries}\label{sec:prelim}

We keep the notation of \cite{paper1} and recall here everything that is
used.  A \emph{standard triangle} is a lattice triangle equivalent to
$\conv\{0,e_1,e_2\}$, and a \emph{unit segment} is a lattice segment of
lattice length one; a \emph{unit parallelogram} is a lattice
parallelogram equivalent to the unit square, and we use the two words
interchangeably, its cone being the ordinary double point.  For a lattice
polygon $Q$ with unit
edges placed at height one, $C(Q)$ denotes the associated Gorenstein
toric affine threefold, and $i(Q)$ the number of interior lattice
points of $Q$.  (Throughout, $C(\cdot)$ \emph{with an argument} denotes
this affine germ; exceptional curves are written $\gamma$, $\gamma_i$,
$\gamma_F$, so that the two conventions in force here, Altmann's
$C(\cdot)$ for the cone and Friedman's $C_i$ for the exceptional curves
\cite{Friedman86,BK10}, do not collide.)  The reduced miniversal base
of $C(Q)$ has one smooth
component per maximal Minkowski decomposition of $Q$
\cite{Altmann97}, the smoothing components corresponding to
decompositions into unit segments and standard triangles \cite{CFP};
$Q$ is of \emph{type (S)} when such a decomposition exists.  The four
imported statements used below are the following; we give each a number,
and refer to it by that number rather than by name.

\begin{theorem}[the trichotomy; {\cite[Thm.~2.2]{paper1}}, packaging
\cite{Altmann97,CFP}]\label{thm:trich}
Let $Q$ be a unit-edge lattice polygon that is not a standard triangle, so
that $C(Q)$ is singular.  Then $Q$ is of exactly one of three types:
\textup{(S)}, admitting a Minkowski decomposition into unit segments and
standard triangles, so that $C(Q)$ is smoothable; \textup{(R)}
\emph{(reduced-rigid)}, admitting no non-trivial Minkowski decomposition, so that
$C(Q)$ has no positive-dimensional deformations (the reduced miniversal
base is a point; first-order deformations may exist); or \textup{(D)}, admitting a
non-trivial decomposition but none into unit segments and standard
triangles, so that $C(Q)$ deforms but is not smoothable.
\end{theorem}

\begin{lemma}[local-to-global obstruction;
{\cite[Lemma~4.1]{paper1}}]\label{lem:ltg}
A deformation of $X$ over a reduced irreducible base induces, near each
singular point, a deformation of the germ.  Hence if the miniversal base
of some germ of $X$ has no smoothing component, then no such deformation
of $X$ (in particular no smoothing) has smooth general fibre.
\end{lemma}

\begin{lemma}[the cascade; {\cite[Lemma~2.4]{paper1}}]\label{lem:cascade}
Let $Q = R_0 + \dots + R_r$ be a Minkowski decomposition of a unit-edge
polygon into lattice summands, and consider
Altmann's homogeneous deformation of $C(Q)$ attached to it, realized on
the Cayley cone $\Cone\bigl(\bigcup_i R_i\times\{f_i\}\bigr)$.  Its
general fibre has exactly one singular point of type $C(R_i)$ for each
two-dimensional non-smooth summand $R_i$, and no others.  In particular
the general fibre is \emph{smooth} when every summand is a unit segment
or a standard triangle.
\end{lemma}

\begin{proposition}[the singular locus of $X$;
{\cite[Prop.~3.1]{paper1}}]\label{prop:planting}
Let $\Delta$ be reflexive with unit edges and let $X \in |{-K}|$ be
generic.  Then $X$ is a Calabi--Yau threefold with Gorenstein canonical
singularities, and for each two-dimensional face $F$ with $C(F)$
\emph{singular}, $X$ meets the toric curve $V(\sigma_F)$ in exactly
$\ell(F^\circ)$ points, all in the open orbit, at each of which an
implicit-function argument on the chart
$U_{\sigma_F} \cong C(F)\times\CC^\ast$ identifies the germ of $X$ with
$C(F)$.  Away from these points $X$ is smooth.
\end{proposition}

Throughout, $\Delta$ is the \emph{fan} polytope, Batyrev's $\Delta^\ast$,
so $\PP_\Delta$ is the toric fourfold of its face fan and our
$\Delta^\circ$ is his $\Delta$, the Newton polytope of the anticanonical
sections.  We use the standard normalization
$\Delta^\circ = \{u : \langle u, v\rangle \ge -1 \ \forall v \in
\Delta\}$, so that a facet of $\Delta$ corresponds to the vertex $u$ of
$\Delta^\circ$ with $\langle u,\cdot\rangle = -1$ on the facet, and the
dual edge of a two-dimensional face $F$ is
$F^\circ = \conv\{u_1,u_2\} \subset \Delta^\circ$ for the two facets
through $F$.  (\cite{paper1} and the ancillary programs list facet data in
the opposite normalization $\langle \cdot,\cdot\rangle = +1$; lattice
lengths are unaffected.)

\section{Homogeneous deformations attached to two-dimensional faces}
\label{sec:datum}

Petracci \cite{Petracci18} constructs, from a \emph{deformation datum}
for a cone, homogeneous deformations of affine toric varieties
(\cite[Def.~3.1, Thm.~3.5]{Petracci18}, after Mavlyutov \cite{Mav04}) and
of polarized projective toric pairs (\cite[Thm.~6.1]{Petracci18}).
Here $N$ is a lattice, $M = \operatorname{Hom}(N,\ZZ)$ its dual, and
$\sigma \subseteq N \otimes \mathbb{R}$ a cone.  A
$\partial$-deformation datum for $(N,\sigma)$, with $\sigma$ strongly convex and
full-dimensional, is a tuple $(Q, Q_0, \dots, Q_r, w)$, $w \in M$,
with (i) $Q \subseteq \sigma$ a non-empty rational polyhedron, (ii)
$0 \notin Q$, (iii)
$Q = Q_0 + \dots + Q_r$, (iv) $Q_1, \dots, Q_r$ lattice polyhedra, (v)
$\min_Q w \ge -1$, and (vi) every vertex of
$\sigma \cap \{\langle w, \cdot\rangle = -1\}$ contained in
$\mathbb{R}^+ Q$.  This is the boundary form of the definition, which is
what \cite[Thm.~6.1]{Petracci18} takes as input.

\begin{lemma}[$\partial$-deformation datum from a codimension-two face]\label{lem:datum}
Let $\Delta \subset \ZZ^4$ be reflexive, $F$ a two-dimensional face with
a Minkowski decomposition $F = Q_0 + \dots + Q_r$ into lattice polytopes,
and let $\tau = \Cone(\Delta\times\{1\}) \subset \ZZ^5$, the cone over the
polarized pair $(\PP_\Delta, \partial\PP_\Delta)$, where $\PP_\Delta$ is the
toric fourfold of the \emph{face fan} of $\Delta$ as in \S\ref{sec:intro}.  Let
$u_1, u_2 \in \Delta^\circ$ be the vertices dual to the two facets of
$\Delta$ containing $F$, so that the dual edge is
$F^\circ = \conv\{u_1, u_2\}$, and suppose $\ell(F^\circ) \ge 2$.  Then
for every \emph{interior} lattice point $v^\ast$ of $F^\circ$, the tuple
\[
\bigl(F\times\{1\},\ Q_0\times\{1\},\ Q_1\times\{0\}, \dots,
Q_r\times\{0\},\ w = (v^\ast, 0)\bigr)
\]
is a $\partial$-deformation datum for $(\ZZ^5, \tau)$ in the sense of
\textup{\cite[Def.~3.1]{Petracci18}}.
\end{lemma}

\begin{proof}
Conditions (i)--(iv) are immediate ($Q = F\times\{1\}$ lies in $\tau$ at
height one, and the summands may be translated so that $Q_1,\dots,Q_r$
are lattice polytopes and $Q_0$ carries the height).  For (v) and (vi),
evaluate $w$ on the rays of $\tau$, which are $(V,1)$ for the vertices
$V$ of $\Delta$: $\langle w, (V,1)\rangle = \langle v^\ast, V\rangle
\ge -1$ because $v^\ast \in \Delta^\circ$, with equality if and only if
$V$ lies on the face of $\Delta$ dual to the face of $\Delta^\circ$
containing $v^\ast$ in its relative interior.  Since $v^\ast$ is an
interior lattice point of the edge $F^\circ$, that dual face is exactly
$F$: equality holds precisely at the vertices of $F$.  Hence
$w \ge -1$ on the primitive generators of the rays of $\tau$, with
equality exactly for those in $\Cone(F\times\{1\})$; in particular
$w \equiv -1$ on $Q = F\times\{1\}$, so $\min_Q w = -1$, which is (v).
Moreover the
vertices of the slice $\tau \cap \{w = -1\}$ are the points $(V,1)$,
$V \in \operatorname{vert}(F)$: a vertex of the slice lies on a ray of
$\tau$, whose $w$-value is an integer $\ge -1$ and negative only if
equal to $-1$.  So (vi) holds with
$\mathbb{R}^+ Q = \Cone(F\times\{1\})$.
\end{proof}

\begin{remark}
At the two endpoints of $F^\circ$ the slice grows to a full facet of
$\Delta$ and condition (vi) would require a Minkowski decomposition of
that three-dimensional facet; the interior points of the dual edge are
exactly what makes the face-local datum possible.  This is the source of
the threshold $\ell(F^\circ) \ge 2$ in Theorem~\ref{thm:A}, and
\S\ref{sec:node} shows the threshold is sharp.
\end{remark}

By \cite[Thm.~6.1]{Petracci18} the datum of Lemma~\ref{lem:datum} yields
a flat deformation
$(\mathcal{X}, \mathcal{D}) \to T$, $T$ a neighbourhood of
$0 \in \mathbb{A}^r$ \textup(cf.\ \cite[Rem.~3.6]{Petracci18}\textup), of the pair
$(\PP_\Delta, \partial\PP_\Delta)$, cut out by explicit trinomials in the
Cox coordinates of a projective toric variety built from the extended
cone $\tilde\tau = \Cone\bigl(\tau,\ Q_0\times\{1\} - e_1 - \dots - e_r,\
Q_1 + e_1, \dots, Q_r + e_r\bigr) \subset \ZZ^{5+r}$.  We call this the
\emph{face family} of $(F, \sum Q_i, v^\ast)$.

\begin{lemma}[chart identification]\label{lem:ident}
Localized at the points of $X$ over $F$, the face family induces on the
germ $C(F)$ the Altmann--Cayley family of the decomposition
$F = Q_0 + \dots + Q_r$ \textup{(}times a smooth trivial
factor\textup{)}; in particular its general fibre near these points is
the general fibre of Lemma~\ref{lem:cascade}.
\end{lemma}

\begin{proof}
Write $N = \ZZ^4$ for the fan lattice of $\PP_\Delta$, let
$\sigma_F = \Cone(F) \subset N_\RR$ be the cone of the face fan over
$F$, and let $N_F$ be the saturated sublattice spanned by $\sigma_F$, so
that $N_F \cong \ZZ^3$, $M_F = \operatorname{Hom}(N_F,\ZZ)$, and
$C(F) = \Spec\CC[\sigma_F^\vee \cap M_F]$ as in \cite[\S2]{paper1}.
(In Lemma~\ref{lem:datum} the same face appeared as
$\Cone(F\times\{1\}) \subset \ZZ^5$, on the cone over the pair; on
passing to the projective toric variety and its charts, $\ZZ^5$ is
replaced by the fan lattice $N$, and the auxiliary lattice
$\ZZ^5 \oplus \ZZ^r$ of the extended cone by $N \oplus \ZZ^r$.)  A
choice of complement, $N = N_F \oplus \ZZ\zeta$, gives the chart
$U_{\sigma_F} \cong C(F)\times\CC^\ast$ of
Proposition~\ref{prop:planting}, with $\CC^\ast$-coordinate the
character dual to $\zeta$.  The choice matters, and we fix $\zeta$ with
\[
N = N_F \oplus \ZZ\zeta, \qquad \langle v^\ast, \zeta\rangle = 0 .
\]
Such $\zeta$ exists: $v^\ast$ restricts to a surjection $N_F \to \ZZ$
(it takes the value $-1$ at each vertex of $F$), so any generator
$\zeta_0$ of $N/N_F$ may be corrected to $\zeta = \zeta_0 - n$ with
$n \in N_F$ chosen so that
$\langle v^\ast, n\rangle = \langle v^\ast, \zeta_0\rangle$.

\emph{Step 1: reduction to the affine germ.}  At the vertices $V$ of
$F$ the functional $v^\ast$ takes the value $-1$ (proof of
Lemma~\ref{lem:datum}), so $v^\ast|_{N_F} = -u_F$, the \emph{negative}
of the Gorenstein degree of $\sigma_F$ (which is exactly the degree
in which the deformations of an isolated Gorenstein toric threefold
singularity live \cite{Altmann95}, cf.\ \cite[Rem.~3.7]{Petracci18}),
and $(F, Q_0,\dots,Q_r, -u_F)$ is a deformation datum for
$(N_F,\sigma_F)$.

We first normalize the translations of the summands.  Choose a vertex
$v=q_0+\dots+q_r$ of $F$, where $q_i$ is the vertex of $Q_i$ exposed by
a functional exposing $v$, and replace
\[
 Q_0\quad\text{by}\quad Q_0+q_1+\dots+q_r,
 \qquad
 Q_i\quad\text{by}\quad Q_i-q_i\quad (i\geq 1).
\]
These integral translations sum to zero and hence preserve the Minkowski
sum.  Since the differences of points of every $Q_i$ lie in
$\operatorname{span}_{\RR}(F-F)=\ker u_F$, after this normalization
\[
 u_F(Q_0)=1,\qquad u_F(Q_i)=0\quad(i\geq1).
\]
It does not change the family up to isomorphism.  Indeed, putting
$a_i=-q_i$ for $i\geq1$, the lattice shear
\[
 (n,k;z_1,\dots,z_r)\longmapsto
 \bigl(n+\textstyle\sum_i z_i a_i,k;z_1,\dots,z_r\bigr)
\]
fixes the original cone $\tau$ and the base characters $e_i^\ast$, and
carries the old extended cone to the one formed from the translated
summands.  We henceforth suppress the primes.

Write $\tilde\sigma_F = \tilde\tau \cap
(N_F\oplus\ZZ^r)_\RR$ for the extended cone, where $N_F$ is embedded in
$\ZZ^5$ by the height section $n \mapsto (n,u_F(n))$; the displayed
normalization is exactly what places the split summand generators in this
graph lattice.  Restrict the face family to the chart at the face
$G\preceq\tilde\tau$ lying over $\sigma_F$ (the cone of Step~2).  To see
that it is a face, use the functional
\[
 L=(v^\ast,1,0,\dots,0).
\]
On a ray $(V,1)$ of $\tau$ it has value
$\langle v^\ast,V\rangle+1\geq0$, with equality exactly for $V\in F$.
On the normalized split generators it vanishes because
$v^\ast|_{N_F}=-u_F$, $u_F(Q_0)=1$, and $u_F(Q_i)=0$ for $i\geq1$.
Thus its zero cone is precisely $G$: each ray over a vertex
$v=q_0+\dots+q_r$ of $F$ is the sum of the corresponding split
generators.  On the chart's torus the $j$-th trinomial of the family reads
\[
\chi^{e_j^\ast} - 1 \;=\; t_j\,\chi^{\tilde w},
\qquad \tilde w|_N = v^\ast
\]
(divide the trinomial of Step~3 by the unit $\chi^{m_{j,0}}$; that the
$t_j$-coefficient carries the character of the datum weight $w$ is
\cite[Thm.~3.5 and Rem.~3.7]{Petracci18}).  Every exponent here
vanishes on $\zeta$: the $e_j^\ast$ because they annihilate $N$, and
$\tilde w$ because $\langle v^\ast, \zeta\rangle = 0$ by the choice
above.  All $r$ equations are therefore pulled back along the projection
$(N_F \oplus \ZZ^r) \oplus \ZZ\zeta \to N_F \oplus \ZZ^r$, i.e.\
constant in the $\CC^\ast$-direction: the restriction of the family
$(\mathcal X,\mathcal D)\to\mathbb A^r$ to this chart is
\[
\bigl(\text{Petracci's affine family of }(N_F,\sigma_F,\ \textstyle F=\sum Q_i)\bigr)\ \times\ \CC^\ast ,
\]
as in the facet case \cite[Prop.~2.12]{CHP24} (there the relevant cone
is maximal, so no complement need be chosen).  It therefore suffices to
identify the affine family, a deformation of $C(F)$ over
$\mathbb A^r$.

\emph{Step 2: the Cayley cone.}  Each vertex of $F$ has a unique
expression $q_0+q_1+\dots+q_r$ with $q_i\in\operatorname{vert}(Q_i)$:
the normal cone of a vertex of a Minkowski sum is the intersection of the
normal cones of its summand faces, and it is two-dimensional, forcing
every summand face to be a vertex (no hypothesis on the summands is
used, a point Proposition~\ref{prop:cascade} exploits); correspondingly the rays of
$\tilde\sigma_F$ over $F$ are the vectors $(q_0, h_0)$ and $(q_j, h_j)$
$(j=1,\dots,r)$, where in the auxiliary coordinates
$(\text{height}, e_1,\dots,e_r)\in\ZZ^{1+r}$ one has
$h_0 = (1;-1,-1,\dots,-1)$ and $h_j = (0; \mathbf e_j)$.  The
$(r{+}1)\times(r{+}1)$ matrix with rows $h_0,h_1,\dots,h_r$ is
$\left(\begin{smallmatrix} 1 & -\mathbf 1^{\!\top}\\ 0 & I_r\end{smallmatrix}\right)$,
which is unimodular; the inverse automorphism of $\ZZ^{1+r}$ sends
$h_i$ to the standard basis vector $f_i$, and hence carries the face
$G = \Cone\bigl(\bigcup_i Q_i\times\{h_i\}\bigr)$ onto the standard
Cayley cone $\Cone\bigl(\bigcup_i Q_i\times\{f_i\}\bigr)$ of the
decomposition (for $r=1$ this is the matrix
$\left(\begin{smallmatrix}1&-1\\0&1\end{smallmatrix}\right)$).

\emph{Step 3: the family is Altmann's.}  In the normalized graph
lattice, extend $u_F$ by zero on the auxiliary directions and put
\[
 a_0=\chi^{u_F},\qquad a_j=\chi^{u_F+e_j^\ast}
 \quad(1\le j\le r).
\]
The exponents take value one on the $Q_0$ and $Q_j$ generators,
respectively, and zero on all other summand generators.  Thus these
characters are regular on the chart of $G$, and Step~2 identifies them
with the Cayley height functions $\chi^{f_0^\ast}$ and
$\chi^{f_j^\ast}$.  Since the normalized datum weight restricts to
$-u_F$, multiplying the torus equation
$\chi^{e_j^\ast}-1=t_j\chi^{-u_F}$ by $a_0$ gives
\[
 a_j-a_0=t_j \qquad(1\le j\le r).
\]
These are Altmann's equations for the deformation attached to
$F=\sum_iQ_i$ \cite{Altmann95}; the base coordinates are differences of
regular height functions.  Their signs are opposite to the convention
in Lemma~\ref{lem:cascade}, which changes only the signs of the
parameters.  Writing $m_{j,0}=u_F$ and $m_{j,1}=u_F+e_j^\ast$, the
same equations read $\chi^{m_{j,1}}-\chi^{m_{j,0}}=t_j$.  (For the type-\textup{(S)} pentagon
$F = T + s$, the Minkowski sum of a standard triangle $T$ and a unit
segment $s$, the case $r = 1$ is worked out explicitly in the ancillary
files:\footnote{\texttt{paper2\_check.py}.} the five vertices of
$F$ split uniquely, and $m_{1,0} = (0,0,1,0,0,0)$,
$m_{1,1} = m_{1,0} + e_1^\ast$.)

\emph{Step 4.}  By \cite[Lemma~2.4]{paper1} the general fibre of this
family is the general fibre of the Altmann family of the decomposition:
one singular point of
type $C(Q_i)$ for each two-dimensional non-smooth summand, and no other
singular points.  This proves the lemma.
\end{proof}

\section{Smoothing over one face, and over several}\label{sec:proofA}

\begin{theorem}[restatement of Theorem~\textup{\ref{thm:A}}]\label{thm:A2}
Let $\Delta$ be reflexive with unit edges and let $F$ be a
two-dimensional face of type \textup{(S)}, with a Minkowski decomposition
$F = Q_0 + \dots + Q_r$ into unit segments and standard triangles, and
suppose $\ell(F^\circ) \ge 2$.  Then there is an explicit $r$-parameter
deformation of the pair $(\PP_\Delta, \partial\PP_\Delta)$, given by
trinomial equations in Cox coordinates, whose induced deformation of the
generic anticanonical $X$ has generic fibre smooth in a neighbourhood of
the $\ell(F^\circ)$ points of type $C(F)$.  In particular, if $F$ is the
\emph{unique} non-smooth two-dimensional face of $\Delta$, then $X$ is
smoothable.
\end{theorem}

We first record that the hypersurface follows the ambient family.

\begin{lemma}[the anticanonical family]\label{lem:sections}
Let $(\mathcal{X}, \mathcal{D}) \to (T,0)$, $T \subseteq \mathbb{A}^r$ a
neighbourhood of the origin, be the face family of
Lemma~\textup{\ref{lem:datum}}, with central fibre
$(\PP_\Delta, \partial\PP_\Delta)$.  Then for every section
$s_0 \in H^0(\PP_\Delta, -K)$ there is, after shrinking $T$, a section
$s \in H^0\bigl(\mathcal{X}, \mathcal{O}_{\mathcal{X}}(\mathcal{D})\bigr)$
restricting to $s_0$, and its zero locus $\mathcal{Y} \subset \mathcal{X}$
is flat over $T$ with $\mathcal{Y}_0 = \{s_0 = 0\}$.
\end{lemma}

\begin{proof}
By \cite[Thm.~6.1]{Petracci18} the boundary deforms along the family, so
$\mathcal{D}\subset\mathcal{X}$ is flat over $T$ with central fibre
$\partial\PP_\Delta$; it is a relative effective Cartier divisor after
shrinking $T$, since $\mathcal{D}_0$ is Cartier ($\PP_\Delta$ being
Gorenstein), so by flatness and Nakayama the ideal sheaf of $\mathcal{D}$
is invertible along $\mathcal{D}_0$, with local generators restricting to
nonzerodivisors on nearby fibres.

The nearby $\mathcal{D}_t$ are again anticanonical.  Gorenstein-ness is
open in a proper flat family, so after shrinking $T$ the sheaf
$\omega_{\mathcal{X}/T}$ is a line bundle, and
$\mathcal{M} = \omega_{\mathcal{X}/T}(\mathcal{D})$ has
$\mathcal{M}_0 \cong \mathcal{O}$ with $h^1(\PP_\Delta,\mathcal{O}) = 0$;
by cohomology and base change the trivializing section extends, and by
properness its zero locus, missing $\mathcal{X}_0$, misses the nearby
fibres.  Hence
$\omega_{\mathcal{X}_t}(\mathcal{D}_t) \cong \mathcal{O}_{\mathcal{X}_t}$,
and $\mathcal{L} = \mathcal{O}_{\mathcal{X}}(\mathcal{D})$ is a relative
anticanonical bundle.

Since $-K_{\PP_\Delta}$ is ample on the Gorenstein Fano toric fourfold
$\PP_\Delta$, toric Kodaira vanishing gives $h^i(\PP_\Delta, -K) = 0$ for
$i>0$, hence the same on nearby fibres by semicontinuity; as
$\chi(\mathcal{X}_t, \mathcal{L}_t)$ is constant in a flat family, so is
$h^0(\mathcal{X}_t, \mathcal{L}_t)$, and Grauert's criterion makes
$\pi_\ast\mathcal{L}$ locally free with formation commuting with base
change.  If $s_0=0$, choose the zero extension; then
$\mathcal Y=\mathcal X$, which is flat over $T$.  Suppose henceforth
that $s_0\ne0$, and extend it to a section $s$ of $\mathcal L$.
Its zero locus $\mathcal Y$ is flat over $T$ by the slicing criterion
\cite[Thm.~22.6]{Matsumura}: it suffices that each $s_t$ be a
nonzerodivisor on $\mathcal X_t$.  After shrinking $T$ the fibres
$\mathcal X_t$ are reduced (geometric reducedness is open in a proper
flat family), and $s_t$ vanishes on no irreducible component: a
component in its zero locus would give $\dim \mathcal Y_t = 4$, while
the locus $\{t : \dim \mathcal Y_t \ge 4\}$ is closed
($\mathcal Y \to T$ is proper) and misses $0$, since
$\dim \mathcal Y_0 = 3$.  Each
$\mathcal{Y}_t$ is an anticanonical Cartier divisor in the Gorenstein
$\mathcal{X}_t$, so adjunction gives
$\omega_{\mathcal{Y}_t} \cong \mathcal{O}$, and
$h^1(\mathcal{Y}_t,\mathcal{O}) = 0$ near $t=0$ by semicontinuity
(the central vanishing follows from the anticanonical divisor
sequence and toric vanishing).  When $s_0$ is generic, as in the
applications below, $\mathcal Y_0$ is normal with canonical
singularities; after shrinking, the fibres are Calabi--Yau threefolds.
\end{proof}

The step that actually smooths the \emph{hypersurface}, as opposed to the
ambient germ, is the following.  It is where the genericity of $s_0$ is
used.

\begin{lemma}[the hypersurface follows the ambient smoothing]
\label{lem:transverse}
Let $F$ be of type \textup{(S)} with decomposition
$F = Q_0 + \dots + Q_r$ into unit segments and standard triangles and
$\ell(F^\circ)\ge2$, let $(\mathcal X,\mathcal D)\to T$ be the
corresponding face family, and let $\mathcal Y \to T$ be the anticanonical
family of Lemma~\textup{\ref{lem:sections}} through a generic
$s_0 \in H^0(\PP_\Delta,-K)$.  Let $p$ be one of the $\ell(F^\circ)$
points of $X = \mathcal Y_0$ lying over $F$.  Then there are a
neighbourhood $U$ of $p$ in $\mathcal X$ and a non-empty Zariski-open
$T^\circ \subseteq T$ such that $\mathcal Y_t \cap U$ is smooth for every
$t \in T^\circ$.
\end{lemma}

\begin{proof}
By Lemma~\ref{lem:ident} there are a neighbourhood $U$ of $p$ in
$\mathcal X$ and an isomorphism over $T$
\[
U \;\cong\; \mathcal{C} \times \CC^\ast ,
\]
where $\mathcal C \to T$ is Altmann's Cayley family of the decomposition
$F = \sum_i Q_i$ and the $\CC^\ast$ factor is the direction of the toric
curve $V(\sigma_F)$; the central fibre is $C(F)\times\CC^\ast$, and $p$
corresponds to $(o, z_0)$ with $o$ the cone point of $C(F)$.  Write $z$
for the coordinate on the $\CC^\ast$ factor.

Because $s_0$ is generic, Proposition~\ref{prop:planting} gives
$\ell(F^\circ)$ \emph{distinct} zeros of $s_0$ on
$V(\sigma_F) \cong \PP^1$; a section of a degree-$\ell(F^\circ)$ line
bundle on $\PP^1$ with $\ell(F^\circ)$ distinct zeros has only simple
ones, so (trivializing $\mathcal L$ near $p$ and regarding $s_0$ as a
holomorphic function)
$\partial s_0/\partial z \,(p) \neq 0$.  By Weierstrass division in the
variable $z$ (valid with parameters over the complex space $C(F)$),
$s_0 = (\text{unit})\cdot(z - \varphi_0)$ near $p$ for a holomorphic
$\varphi_0 : W_0 \to \CC^\ast$ on a neighbourhood $W_0$ of $o$ in
$C(F)$, so $\{s_0 = 0\}$ is there the graph of $\varphi_0$; this is
how Proposition~\ref{prop:planting} identifies the germ of $X$ at $p$
with $C(F)$.

Non-vanishing of $\partial s_t/\partial z$ is an open condition on
$U \times T$, and $s_t$ depends holomorphically on $t$
(Lemma~\ref{lem:sections}), so after shrinking $U$ and $T$ we have
$\partial s_t/\partial z \neq 0$ throughout $U$ for every $t$.
Weierstrass division applies fibrewise: near each of its points,
$\mathcal Y_t \cap U$ is the graph of a holomorphic function on an open
subset of $\mathcal C_t$; in particular, locally on $U$,
\[
\mathcal Y_t \cap U \;\cong\; \mathcal{C}_t \cap (\text{an open set}) .
\]
So $\mathcal Y_t \cap U$ is smooth as soon as $\mathcal C_t$ is.  Every
summand $Q_i$ being a unit segment or a standard triangle, all summand
cones $C(Q_i)$ are smooth, so by Lemma~\ref{lem:cascade} the general
fibre $\mathcal C_t$ is smooth; this holds on a non-empty Zariski-open
$T^\circ \subseteq T$.
\end{proof}

\begin{proof}[Proof of Theorem~\textup{\ref{thm:A2}} (assembly)]
Let $(\mathcal{X}, \mathcal{D}) \to T$ be the face family and
$\mathcal{Y} \to T$ the anticanonical family of
Lemma~\ref{lem:sections} through a generic $s_0$, so
$X = \mathcal{Y}_0$.

\emph{The points over $F$ are smoothed.}  This is
Lemma~\ref{lem:transverse}, applied at each of the $\ell(F^\circ)$ points
of $X$ over $F$: intersecting the finitely many non-empty Zariski-open
subsets $T^\circ$ so obtained, $\mathcal{Y}_t$ is smooth near all of them
for generic small $t \in T$.

\emph{The unique-face case.}  Suppose $F$ is the unique non-smooth
two-dimensional face of $\Delta$.  Then $X$ is smooth away from the
$\ell(F^\circ)$ points over $F$ (Proposition~\ref{prop:planting}), so
$\mathcal{Y}_0$ is smooth outside the former $F$-points; smoothness is
open on the total space of the proper flat family $\mathcal{Y}\to T$, so
after shrinking, $\mathcal{Y}_t$ is smooth away from a neighbourhood of
those points for all small $t$, and smooth near them for generic $t$ by
the previous step.  Restricting $\mathcal{Y}$ to a generic line through
$0 \in T$ gives a flat proper family over a disc with special fibre $X$
and smooth Calabi--Yau general fibre (Lemma~\ref{lem:sections}): $X$ is
smoothable.
\end{proof}

\begin{remark}[first-order persistence at the other faces]
\label{rem:persist}
None of the results below use this remark; we record it because it
explains why the face families of distinct faces do not interfere at
first order, which is the evidence behind Question~\ref{q:multi}.  The
face family does not disturb the other germs at first order.  By
\cite[Rem.~3.7]{Petracci18} (stated there for the affine construction
of \cite[Thm.~3.5]{Petracci18}; the projective family carries the same
parameter weights, as its trinomials do), the Kodaira--Spencer image of
the family of
Lemma~\ref{lem:datum} lies in the weight-$w$ component of $T^1$, with
$w = (v^\ast,0)$.  For a face $F' \neq F$, the space $T^1_{C(F')}$ is
concentrated in the degree $-u_{F'}$, the character taking the value
$-1$ on all rays $(V',1)$ of $\sigma_{F'}$ \cite{Altmann95}; but $F'$
has a vertex $V'$ not on $F$, where
$\langle w, (V',1)\rangle = \langle v^\ast, V'\rangle \ge 0 \neq -1$
(it is an integer $\ge -1$, and $= -1$ exactly on $F$ by the proof of
Lemma~\ref{lem:datum}).  So the induced first-order deformation of the
germ over $F'$ vanishes; likewise no positive integer multiple of $w$
equals $-u_{F'}$, which by Rim's equivariant structure on the miniversal
deformation \cite{Rim} forces the higher-order Taylor coefficients of an
equivariant classifying map into vanishing weight spaces as well.
Upgrading this to constancy of the germs requires identifying the
induced deformation of the affine chart along $V(\sigma_{F'})$ with an
equivariant product family, as in Lemma~\ref{lem:ident}; this is
routine for
faces $F'$ disjoint from $F$, where the chart of the extended cone over
$\sigma_{F'}$ is untouched by the construction, but we do not pursue it
here.
\end{remark}

\begin{theorem}[restatement of Theorem~\textup{\ref{thm:Amulti}}]\label{thm:A3}
Let $\Delta$ be reflexive with unit edges and suppose that every
non-smooth two-dimensional face $F_1,\dots,F_s$ of $\Delta$ is of type
\textup{(S)} (a Minkowski sum of unit segments and standard
triangles), with $\ell(F_i^\circ) \ge 2$.  Then:
\begin{enumerate}
\item for every singular point $p$ of $X$ there is an explicit
deformation of $X$ whose generic fibre is smooth in a neighbourhood of
$p$; consequently the locus $D_p \subset (B,0)$ of points of the
semiuniversal deformation of $X$ whose fibres remain singular near $p$
is a proper closed analytic subgerm;
\item if the base germ $(B,0)$ of the semiuniversal deformation of $X$
is irreducible (for instance smooth), then $X$ is smoothable;
\item if every $F_i$ is a unit parallelogram, then $X$ is smoothable
unconditionally.
\end{enumerate}
\end{theorem}

\begin{proof}
(i)  Let $p$ lie over the face $F_i$ and apply Theorem~\ref{thm:A2} to
$F_i$: the resulting family $\mathcal{Y}\to T_i$ has generic fibre
smooth near $p$.  A semiuniversal deformation
$\mathcal{S} \to (B,0)$ of the compact complex space $X$ exists by
Grauert \cite{Grauert74}; choose a representative, proper over $B$.
The relative singular locus $\Sigma = \Sing(\mathcal{S}/B)$ is closed,
meets the central fibre in the finite set $\Sing X$, and so (fibre
dimension being upper semicontinuous and the map proper) is finite
over $B$ after shrinking.  Then $\Sigma$ decomposes as a disjoint union
$\bigsqcup_q \Sigma_q$ over the singular points $q$ of $X$, and each
$D_q := \operatorname{im}(\Sigma_q \to B)$ is a closed analytic subgerm
(finite maps are proper).  The family $\mathcal{Y} \to T_i$ is induced
from $\mathcal{S}$ by a classifying map $\gamma_i : (T_i,0) \to (B,0)$;
for generic $t$ the fibre is smooth near $p$, so
$\gamma_i(t) \notin D_p$: the subgerm $D_p$ is proper.

(ii)  After the shrinking above, a fibre $\mathcal{S}_b$ is smooth if and
only if $b \notin D := \bigcup_q D_q$: away from the $\Sigma_q$ the total
space is smooth over $B$ because the central fibre is (openness of
smoothness).  If $(B,0)$ is irreducible, $D$, a finite union of proper
closed analytic subgerms, is a proper closed analytic subgerm, so its
complement meets every neighbourhood of $0$.  By the curve selection
lemma for analytic germs there is a holomorphic arc
$\varphi : (\mathbb{D},0) \to
(B,0)$ with $\varphi(\mathbb{D}\setminus 0) \not\subset D$
(normalize $B$ and slice by $\dim B - 1$ generic hyperplanes through a
preimage of $0$: the preimage of $D$ has smaller dimension, so the
resulting curve germ is not contained in it); an
irreducible arc not contained in the closed analytic set $D$ meets it
in a discrete set, so after shrinking the disc $\varphi(t) \notin D$
for all $t \neq 0$, and $\varphi^\ast\mathcal{S}$ is a smoothing of $X$.

(iii)  If every $F_i$ is a unit parallelogram, $X$ is a nodal
Calabi--Yau threefold and the MPCP resolution $\widehat X \to X$ is a
projective \emph{small} resolution (each square face has no interior or
edge lattice points, so the maximal triangulation inserts only the
diagonal wall).  The $\ell(F_i^\circ) \ge 2$ nodes over one face lie on
the single toric curve $V(\sigma_{F_i}) \cong \PP^1$, and their
exceptional curves in the MPCP resolution all represent the same class
in $H_2(\widehat X;\mathbb{Q})$: in the Batyrev--Kreuzer presentation of
this group by rational linear relations among the vertices of $\Delta$,
each maps to the diagonal relation of the face $F_i$, the sum of the
endpoints of one diagonal minus the sum of the endpoints of the other,
Batyrev--Kreuzer's $\rho_\theta$ \cite{BK10} (see \S\ref{sec:bk}); choosing integers $a_{i,1},\dots,a_{i,\ell_i}$
(writing $\ell_i = \ell(F_i^\circ)$), all
non-zero, with $\sum_j a_{i,j} = 0$ (possible exactly because
$\ell_i \ge 2$), and summing over the faces produces a relation
$\sum a_k [\gamma_k] = 0$ among \emph{all} exceptional classes with every
coefficient non-zero.  By the smoothing criterion for nodal
Calabi--Yau threefolds with projective small resolution, $X$ is
smoothable.  The criterion is Friedman's first-order result
\cite{Friedman86}, upgraded to genuine
smoothings by the unobstructedness of nodal Calabi--Yau deformations
due to Kawamata and Tian (see \cite{RollenskeThomas}; the criterion is
Namikawa's \cite{Namikawa02}, as used by \cite{BK10}).
\end{proof}

\begin{remark}\label{rem:noiterate}
To iterate the construction for several faces, one would need a toric
presentation of the first deformed fibre and a compatible deformation
datum for each remaining face.  The construction above does not provide
these data, so it does not justify a second application of
Lemma~\ref{lem:datum}.  We also do not construct a single datum
combining arbitrary choices on several faces.  This is the same difficulty that leads
Corti--Hacking--Petracci, one dimension down, to assemble their per-facet
admissible Minkowski data not torically but through the semiuniversal
deformation of a partial resolution with unobstructed deformations
\cite{CHP24}.  Whether Theorem~\ref{thm:A3}(ii) holds unconditionally
(equivalently, whether the irreducibility hypothesis can be removed)
is Question~\ref{q:multi} below; note that by Gross \cite{Gross97}
the deformation space of a Calabi--Yau threefold with canonical
singularities can be singular, so the hypothesis is not automatic.
\end{remark}

\subsection{Singularities after an arbitrary Minkowski decomposition}\label{sec:cascade}

Theorem~\ref{thm:A} invoked the type~(S) hypothesis only at the last step, to
conclude that the local general fibre is smooth.  The deformation datum and its
chart identification are insensitive to the summand types, so the same family
realizes globally the \emph{cascade} of \cite[Lemma~2.4]{paper1}: it moves the
singular points over $F$ to the germs of the two-dimensional non-smooth summands
of a chosen Minkowski decomposition.

\begin{proposition}[global cascade]\label{prop:cascade}
Let $\Delta$ be reflexive with unit edges whose \emph{unique} non-smooth
two-dimensional face $F$ has a
Minkowski decomposition $F = Q_0 + \dots + Q_r$ into lattice polytopes and
$\ell(F^\circ) \ge 2$.  Then there is an explicit $r$-parameter deformation of the
generic anticanonical $X$, induced by a trinomial deformation of the ambient
toric pair,
such that on a generic line through the origin of the base, the fibre over
every sufficiently small $t \neq 0$ is a Calabi--Yau threefold that has exactly
$\ell(F^\circ)$ points of type $C(Q_i)$ for each
two-dimensional non-smooth summand $Q_i$, and no other singular points.
In particular, if $F$ is deformable-but-non-smoothable (type
\textup{(D)}), take the construction for a maximal decomposition
$F=R_0+\dots+R_s$.  Then $X$ deforms to a Calabi--Yau threefold in which
each of the $\ell(F^\circ)$ points of type $C(F)$ is replaced by a reduced-rigid
point $C(R_j)$ for each two-dimensional non-smooth summand $R_j$; this
deformed threefold is again non-smoothable.
\end{proposition}

\begin{proof}
Lemma~\ref{lem:datum} builds the deformation datum from any Minkowski decomposition
into lattice polytopes with $\ell(F^\circ) \ge 2$: its proof uses only that
$F = \sum_i Q_i$ and $v^\ast \in \operatorname{relint} F^\circ$, never the summand
types.  Lemma~\ref{lem:ident} is likewise independent of those types: a
vertex of a Minkowski sum
decomposes uniquely into vertices of the summands (its normal cone is the
intersection of theirs), so that step, and the unimodular-height identification
with the Altmann--Cayley cone, hold for an arbitrary decomposition.  By the cascade
lemma \cite[Lemma~2.4]{paper1} the general fibre of the Altmann--Cayley family has
exactly one singular point of type $C(Q_i)$ for each two-dimensional non-smooth
summand $Q_i$, and is smooth elsewhere.  Lemma~\ref{lem:sections} carries a generic
anticanonical section along the family, with Calabi--Yau fibres, and the
graph description in the proof of Lemma~\ref{lem:transverse} (which
uses only $\partial s_t/\partial z \neq 0$, never the summand types)
identifies $\mathcal Y_t$, near each former $F$-point, with an open
subset of the Altmann fibre $\mathcal C_t$.  Away from the former
$F$-points the fibres are smooth, since $X$ is (the face $F$ being the
unique non-smooth face) and smoothness is open on the total space of the
proper flat family.  So the singular points of $\mathcal Y_t$ are
\emph{at most} the summand germs, and it remains to check that they
actually occur: that the singular points of $\mathcal C_t$ do not
escape the neighbourhood over which the graph description holds.  The
Altmann--Cayley family is equivariant for the one-parameter subgroup
through the lattice point $n_0 = (x, 1, \dots, 1)$, $x$ a vertex of $F$,
which lies in the Cayley cone (it is the sum of one ray generator from
each summand family); every base parameter is a difference
of the height characters $\chi^{u_i}$, each of weight
$\langle u_i, n_0\rangle = 1$, so the subgroup carries the fibre over
$t$ to the fibre over $\lambda t$ and
$\operatorname{Sing} \mathcal C_{\lambda t}
= \lambda \cdot \operatorname{Sing} \mathcal C_t$.  Because $n_0$ lies
in the cone, every character of the dual cone has non-negative
$n_0$-weight, so $\lim_{\lambda\to 0}\lambda \cdot p$ exists for every
point $p$ of the affine total space; applied to points of the relative
singular locus, which is closed, the limits lie in
$\operatorname{Sing} C(F) = \{o\}$.  Hence along any fixed direction in
the base the singular points of $\mathcal C_{\lambda t}$ converge to
the cone point as $\lambda \to 0$, and on a generic line through the
origin, for all sufficiently small $t \neq 0$ they lie inside the
neighbourhood of $o$ over which
$\mathcal Y_t$ is a graph (a single sheet over a neighbourhood of $o$
uniform in small $t$, by the Weierstrass division of the proof of
Lemma~\ref{lem:transverse} applied at the central point of the total
space), which transports each of them, with its germ
$C(Q_i)$, to a singular point of $\mathcal Y_t$.
When $F$ is of type~(D), take the decomposition to be maximal.  Every
summand is then Minkowski-indecomposable, so every two-dimensional non-smooth
summand gives a reduced-rigid cone.  Thus each $C(F)$ point
becomes a collection of $C(R_j)$ points; since each $C(R_j)$ has no smoothing
component (Theorem~\ref{thm:trich}), the deformed
threefold admits no
smoothing by Lemma~\ref{lem:ltg}.
\end{proof}

\begin{remark}\label{rem:cascade-global}
The smallest instance is the type-\textup{(D)} pentagon $Q_B = T + s$ of \cite{paper1}
($T$ a non-standard determinant-$3$ lattice triangle, whose cone $C(T)$
is the $\frac13(1,1,1)$ point,
and $s$ a unit segment; its maximal decomposition is $T + s$
with the single rigid summand $T$), realized with $\ell(F^\circ) = 3$ as the unique
non-smooth face of the polytope $\Delta_B$ of \cite{paper1}.  The Calabi--Yau
threefold $X_B$, whose three singular points are all deformable yet
non-smoothable, then deforms, through a family induced by explicit ambient
trinomials, to a
Calabi--Yau threefold with three ordinary $\frac13(1,1,1)$ points, itself
non-smoothable.  (That the conditions of Lemma~\ref{lem:datum} hold for this
non-standard summand is a finite
verification.\footnote{\texttt{cascade\_check.py}.})  Thus the family
replaces each original singularity by its rigid summand singularity,
and both the central and general Calabi--Yau threefolds are nonsmoothable.
\end{remark}

\section{Sharpness: a one-nodal Calabi--Yau threefold with no smoothing}
\label{sec:node}

\begin{theorem}[restatement of Theorem~\textup{\ref{thm:B}}]\label{thm:B2}
Let $\Delta$ be a reflexive $4$-polytope with unit edges.  If the singular
locus of its generic anticanonical hypersurface $X\subset\PP_\Delta$
consists of a single ordinary double point, then $X$ admits no smoothing.

The polytope $\Delta_N$ of Theorem~\textup{\ref{thm:B}} provides an explicit
example: it has $7$ vertices and unit
edges; its unique non-smooth two-dimensional face is the unit square
$\square\times\{(1,0)\}$, of type \textup{(S)}, with dual edge of
lattice length $1$; the generic anticanonical $X_N$ is a Calabi--Yau
threefold with exactly one ordinary double point; and $X_N$ admits no
smoothing.  The MPCP resolution
$\widehat X_N$ has $(h^{1,1}, h^{2,1}) = (3, 103)$ and Euler number
$-200$.
\end{theorem}

\begin{proof}
That $\Delta_N$ is reflexive with unit edges, and that its $21$
two-dimensional faces comprise $20$ standard triangles and the square
with $\ell(F^\circ)=1$, are finite computations; the Hodge numbers are
given by Batyrev's formulas \cite[\S4]{Batyrev94}, evaluated in exact
arithmetic.  The ten vertices of $\Delta_N^\circ$, in the normalization
of \S\ref{sec:prelim} (so that
$\Delta_N = \{v : \langle u, v\rangle \ge -1$ for every vertex $u$ of
$\Delta_N^\circ\}$, with equality exactly on the facet dual to $u$),
are
\[
\begin{array}{lllll}
(0,0,-1,-1), & (0,0,-1,0), & (1,0,-1,-1), & (5,-2,-1,-1), & (0,1,0,-1),\\[2pt]
(0,-2,-1,-1), & (0,-2,-1,4), & (-2,0,1,-1), & (-2,0,1,4), & (-2,5,6,-1),
\end{array}
\]
with the dual edge $F^\circ$ of the square face spanned by the first
two.\footnote{The ancillary programs list the same data in the opposite
normalization $\langle\cdot,\cdot\rangle=+1$ of \cite{paper1}, that is,
negated.}  By
Proposition~\ref{prop:planting}, $X_N$ has exactly one singular point, an ordinary
double point.

It remains to prove the uniform assertion.  Let $X\subset\PP_\Delta$ be
any hypersurface satisfying its hypotheses.  Proposition~\ref{prop:planting}
and the toric description of an ordinary double point identify its unique
singularity with the cone over a unit square whose dual edge has lattice
length one.  The square has no interior lattice points, so an MPCP
subdivision of its cone inserts only a diagonal wall: the induced crepant
resolution $\pi:\widehat X\to X$ is \emph{small}, with exceptional locus a
single rational curve $\gamma$, and it is projective because the MPCP
subdivision is projective \cite{Batyrev94}.  If $A$ is ample on
$\widehat X$ then $A\cdot\gamma>0$, so
$[\gamma] \neq 0$ in $H_2(\widehat X; \mathbb{Q})$.  By Friedman's theorem
\cite{Friedman86,RollenskeThomas}, a nodal
Calabi--Yau threefold with a projective small resolution is smoothable
only if the exceptional classes
satisfy a relation $\sum a_i [\gamma_i] = 0$ in
$H_2(\widehat X;\mathbb{Q})$ with all $a_i \neq 0$; with a
single node this forces $[\gamma] = 0$.  Hence $X$ admits no smoothing,
and the explicit $X_N$ is a special case.
\end{proof}

\begin{remark}\label{rem:nonpoly}
No \emph{polynomial} deformation (moving $X_N$ inside $\PP_{\Delta_N}$)
can smooth the node, since every $\Delta$-regular member has the same
singularity structure by Proposition~\ref{prop:planting}; so
Theorem~\ref{thm:B2} says the non-polynomial directions of \cite{Mav04}
fail as well.
Contrast the generic one-nodal quintic threefold, where the node smooths
polynomially: there the small resolutions are non-K\"ahler and
$[\gamma]$ vanishes rationally.  $\Delta_N$ manufactures the opposite
situation inside the Kreuzer--Skarke list.  Non-smoothable nodal
Calabi--Yau threefolds as such are not new: the local-to-global
failure for nodes is the classical content of \cite{Friedman86}, and
obstructed nodal configurations appear in Namikawa's study of the local
moduli of Calabi--Yau threefolds \cite{Namikawa02}.  What $\Delta_N$
adds is a one-node realization produced, and shown non-smoothable,
by the combinatorics of a single reflexive polytope.
\end{remark}

\begin{remark}[relation to the Batyrev--Kreuzer census]\label{rem:BKdecides}
The nearest relative of Theorem~\ref{thm:B2} is the criterion of
\cite{BK10} itself, and we state the relationship precisely.  Every
two-dimensional face of $\Delta_N$ is a standard triangle or a unit
parallelogram, with at least one parallelogram, so $\Delta_N$ belongs to
the class of $198{,}849$ polytopes surveyed in \S\ref{sec:bk}.  The
Batyrev--Kreuzer criterion recalled there is a necessary \emph{and}
sufficient condition, and it constrains exactly the faces with
$\ell(F^\circ)=1$: the diagonal relation of such a face must lie in the
span of the diagonal relations of the \emph{other} square faces.  The
unique square face of $\Delta_N$ has $\ell(F^\circ)=1$ and there are no
other square faces, so the span is zero while the diagonal relation is
not; their criterion therefore also returns \emph{non-smoothable}, and
$\Delta_N$ is one of the $198{,}849 - 30{,}241 = 168{,}608$ polytopes
their criterion does not show to be smoothable.  Theorem~\ref{thm:B2} is proved
directly from \cite{Friedman86}, without the census, and what it
contributes is not the mechanism but the \emph{instance}: an explicit
minimal one --- one node, seven vertices, Hodge numbers computed ---
exhibited for the single purpose of showing that the threshold
$\ell(F^\circ)\ge2$ of Theorem~\ref{thm:A} cannot be lowered.  We have not
found such an example in the literature: \cite{BK10} tabulate the polytopes
their criterion shows to be smoothable and do not discuss the complement.
\end{remark}

\section{Single del Pezzo cone points of dual-edge length one}
\label{sec:germ}

We recall the consequence of \cite[Thm.~A]{paper5} needed here.
Let $Z$ be a connected normal projective complex threefold with
$\omega_Z\cong\mathcal O_Z$ and $H^1(Z,\mathcal O_Z)=0$, whose only
singularity is analytically the anticanonical cone over
$S\in\{\PP^1\times\PP^1,\mathrm{dP}_7,\mathrm{dP}_6\}$.
Let $\widehat Z\to Z$ be the crepant resolution locally given by
contracting the zero section of $\operatorname{Tot}(K_S)$, and let
$L$ be the link of the singularity.  The circle-bundle projection gives
\[
 H_2(L;\mathbb Q)\cong K_S^\perp
 :=\{\alpha\in H_2(S;\mathbb Q):K_S\cdot\alpha=0\}.
\]
If $Z$ is smoothable, then
\begin{equation}\label{eq:lonekernel}
 \ker\bigl\{K_S^\perp\longrightarrow
                 H_2(\widehat Z;\mathbb Q)\bigr\}\ne0,
\end{equation}
where the map is induced by $S\hookrightarrow\widehat Z$.
Indeed, the cited theorem supplies a nonzero relation in a subspace of
$H_2(L;\mathbb Q)$ corresponding to the chosen local smoothing
component.  Each such subspace is contained in $K_S^\perp$, including
both components of the degree-six cone.  Thus \eqref{eq:lonekernel}
is necessary regardless of the component used.  It is only this
necessary consequence, not a sufficient criterion on the full link
homology, that we use below.

\begin{proposition}[restatement of
Proposition~\textup{\ref{prop:C}}]\label{prop:C2}
Let $\Delta$ be reflexive with unit edges whose unique non-smooth
two-dimensional face $F$ has exactly one interior lattice point, is of
type \textup{(S)}, and has $\ell(F^\circ) = 1$.  Then the singular locus
of $X$ is one point, the anticanonical cone over $S \in
\{\PP^1{\times}\PP^1,\ \mathrm{dP}_7,\ \mathrm{dP}_6\}$, and an MPCP
resolution $\widehat X \to X$ can be chosen whose exceptional locus is
the smooth irreducible surface $S$, contracted to the point.  The
contraction has relative Picard number
$\rho(\widehat X/X) = \rho(S) \in \{2,3,4\}$, so it factors in the projective category and
is \emph{never} a primitive contraction, and $X$ is not
$\mathbb{Q}$-factorial at the point.  The map
$H_2(S;\mathbb Q)\to H_2(\widehat X;\mathbb Q)$ is injective;
the necessary condition of \cite[Thm.~A]{paper5}, recalled in
\S\ref{sec:germ}, therefore implies that $X$ is not smoothable.
\end{proposition}

\begin{proof}
A unit-edge polygon with exactly one interior lattice point is, up to
equivalence, one of five, and the corresponding cones are the
anticanonical cones over the five smooth toric del Pezzo
surfaces (\cite[Rem.~2.6]{paper1}, which rests on
\cite[\S5.5]{Gross97}); two of them, $\PP^2$ and $\FF_1$, give reduced-rigid cones, and the remaining
three are of type~(S): the cones over $\PP^1\times\PP^1$ $(\deg 8)$,
$\mathrm{dP}_7$ $(\deg 7)$ and $\mathrm{dP}_6$ $(\deg 6)$.  So $F$ is
one of these three.  An MPCP subdivision is a maximal \emph{projective}
crepant subdivision and is not unique; we choose one whose restriction
to $\Cone(F)$ is the star subdivision at the interior lattice point $q$
of $F$: a pulling triangulation of the boundary that pulls first at $q$
is regular \cite[\S4.3]{DLRS}, hence projective \cite[Ch.~15]{CLS}, and
pulling at $q$ first forces every maximal cell over $F$ to contain $q$.  The star
subdivision of $\Cone(F)$ at $q$ is unimodular: after translating $q$ to
the origin, each triangle spanned by an edge of $F$ and the origin has
determinant $\pm1$, a finite check for each of the three
polygons.\footnote{\texttt{paper2\_check.py}.}  Hence this MPCP resolution is smooth over
the point, with irreducible exceptional divisor the smooth toric del
Pezzo surface $S$ (the toric surface of the star of $q$) contracted to
the point, and it is an isomorphism elsewhere;
all other faces being smooth and all edges unit, $\widehat X$ is smooth
everywhere.

Now the negative statements.  Let $k$ be the number of vertices of $F$,
so $\rho(S) = k-2 \in \{2,3,4\}$.  For each vertex $v'$ of $F$ the
toric divisor $V(v') \cap \widehat X$ of $\widehat X$ restricts to
$S$ as the toric boundary curve of $S$ corresponding to the edge
$(q,v')$ of the star subdivision; these boundary divisors generate
$\operatorname{Pic}(S)_{\mathbb{R}}$, so the intersection pairing of
divisors of $\widehat X$ with curves contained in $S$ has rank
$\rho(S)$.  Conversely every contracted curve lies in $S$, so its
relative numerical class is in the image of $N_1(S)$ and the relative
Picard number is at most $\rho(S)$.  Thus the classes of the contracted
curves span a face of dimension $\rho(S) \ge 2$ in $N_1(\widehat X)$:
the relative Picard number of $\widehat X \to X$ is $\rho(S) \ge 2$.  Since
$K_{\widehat X} \equiv 0$, the cone theorem says nothing about these
rays directly; apply it instead to the klt pair $(\widehat X,
\varepsilon S)$ for small $\varepsilon > 0$: adjunction gives
$S|_S = K_S$, and $-K_S$ is ample on the del Pezzo surface $S$, so
$(K_{\widehat X} + \varepsilon S)\cdot\gamma = \varepsilon\,
K_S\cdot\gamma < 0$ for every contracted curve $\gamma$ (all of which
lie in $S$).  The relative cone and contraction theorems for the pair
then make $\overline{NE}(\widehat X/X)$ rational polyhedral with
contractible extremal rays, and $\rho(S) \ge 2$ yields a factorization
$\widehat X \to Z \to X$ with both steps non-trivial, so the
contraction is not primitive; and since every
projective crepant resolution of $X$ is connected to this one by flops
\cite{Kawamata08}, which preserve the relative Picard number, no choice
is primitive.
(Concretely, for $S = \mathrm{dP}_7$ or $\mathrm{dP}_6$: a
$(-1)$-curve $\gamma \subset S$ has
$N_{\gamma/\widehat X} = N_{\gamma/S}\oplus\omega_S|_\gamma =
\mathcal{O}(-1)\oplus\mathcal{O}(-1)$ and can be contracted separately.)
Finally, the local class group of the germ $C(F)$ is the cokernel of
$M_F \to \ZZ^{k}$ (one coordinate per ray of $\sigma_F$), of rank
$k - 3 \ge 1$, and these local classes are realized by algebraic
divisors on $X$: under the chart identification of
Proposition~\ref{prop:planting}, the closures of the intersections of $X$ with the toric
divisors of $\PP_\Delta$ through the point restrict to the toric
divisors of $C(F)$, which generate its class group.

We now verify directly that one of these globally realized classes remains
non-torsion analytically.  Let
\[
 U=\operatorname{Tot}(\omega_S)\smallsetminus S.
\]
Contraction of the zero section identifies $U$ with the punctured analytic
germ of $C(F)$.  The Gysin sequence of the associated circle bundle gives
\[
 H^2(U,\mathbb{Q})\cong
 H^2(S,\mathbb{Q})/\mathbb{Q} c_1(K_S),
\]
because $H^1(S,\mathbb{Q})=0$.  The exponential sequence on $S$, together
with $H^1(S,\mathcal O_S)=H^2(S,\mathcal O_S)=0$, identifies
$\operatorname{Pic}(S)_{\mathbb{Q}}$ with $H^2(S,\mathbb{Q})$.  Since $S$ is
a smooth toric del Pezzo surface with $\rho(S)\ge2$, its toric boundary
classes generate this vector space, and at least one of the boundary classes
realized by a divisor on $X$ has first Chern class not in
$\mathbb{Q}c_1(K_S)$.  Its pullback to $U$ consequently has nonzero rational
first Chern class and is not torsion as an analytic line bundle.  If the
corresponding divisor on $X$ were $\mathbb Q$-Cartier at the point, a positive
multiple would be locally principal and its restriction to $U$ would be
trivial, a contradiction.  Thus $X$ is not
$\mathbb{Q}$-factorial.  (The germ computation alone would not suffice:
a generic one-nodal quintic threefold is $\mathbb{Q}$-factorial although
its node is analytically non-factorial.)  Thus
the hypotheses of both \cite[Thm.~5.8]{Gross97} (primitivity) and
\cite[Thm.~4.3]{Gross97} ($\mathbb{Q}$-factoriality of $X$) fail.

The same restriction calculation also determines smoothability.
The boundary divisor classes above span $H^2(S;\mathbb Q)$, so
$H_2(S;\mathbb Q)\to H_2(\widehat X;\mathbb Q)$ is injective.
In particular its restriction to $K_S^\perp$ has zero kernel.
There is only one singular point, so the necessary condition
\eqref{eq:lonekernel} fails and $X$ is not smoothable.
\end{proof}

\begin{remark}\label{rem:germalign}
Gross's theorem \cite[Thm.~5.8]{Gross97} smooths exactly these germs
when they are realized by \emph{primitive} type~II contractions; such
realizations exist on non-toric Calabi--Yau threefolds, where the
restriction $\operatorname{Pic}(\widehat X) \to \operatorname{Pic}(S)$
can have rank one, and are constructed and smoothed for all three germs
by Kapustka--Kapustka and Kapustka
\cite{KapustkaKapustka09,Kapustka09}.  The theorem fails exactly when the
exceptional surface is $\PP^2$ or $\FF_1$, which are precisely the two
classes with one interior point that
are \emph{reduced-rigid} in the trichotomy (Theorem~\ref{thm:trich}): the germ
classification and the global smoothing problem see the same dichotomy.
Proposition~\ref{prop:C2} shows that these Batyrev hypersurfaces do
not provide the primitive realization.  Their nonsmoothability follows
from the injectivity of $H_2(S;\mathbb Q)\to
H_2(\widehat X;\mathbb Q)$ and \eqref{eq:lonekernel}, even though
each possible germ is locally smoothable.
\end{remark}

\section{Conifold configurations and the Batyrev--Kreuzer census}
\label{sec:bk}

Suppose every non-smooth two-dimensional face of $\Delta$ is a unit
parallelogram, so that $X$ is a nodal Calabi--Yau threefold; this is the
setting of Batyrev--Kreuzer \cite{BK10}.  They identify
$H_2(\widehat X;\mathbb{Q})$ with the space of rational linear relations
among the vertices of $\Delta$, and show that the $\ell(F^\circ)$
exceptional curves over one square face $F$ all represent a single class
$[\gamma_F]$ there: up to the sign fixed by the choice of small resolution,
the diagonal relation of $F$, the sum of the endpoints of one
diagonal minus the sum of the endpoints of the other \cite{BK10}.
(The comparison carries a change of polarity: \cite{BK10} place the
parallelogram faces on the polar polytope, so their $1$-dimensional face
$\theta$ is our dual edge $F^\circ$, their $2$-face $\theta^\circ$ is our
$F$, and their relation $\rho_\theta$ is our diagonal relation of $F$.
Our $\ell(F^\circ)$ is their $k_\theta = l(\theta) - 1$, the number of
nodes lying over the face.)
Friedman's relation restricted to one face,
$\sum_{i=1}^{\ell} a_i [\gamma_F] = 0$ with all $a_i \neq 0$, is therefore
satisfiable if and only if $\ell = \ell(F^\circ) \ge 2$; this reproduces
the threshold of Theorem~\ref{thm:A} from homology alone.  A polytope in
this class is thus smoothable whenever all its square faces have
$\ell(F^\circ) \ge 2$ (Theorem~\ref{thm:Amulti}(iii): sum the per-face
relations and apply the nodal smoothing criterion to the projective
small MPCP resolution),
while a face with $\ell(F^\circ) = 1$ requires a relation involving
other faces: the class $[\gamma_F]$ must be a rational combination of the classes over
other faces.

We tabulated this over the per-vertex-count database copy described in
\cite[\S6.1]{paper1}, together with the separately reconstructed
$36$-vertex polytope.  For each reflexive $4$-polytope we test every two-dimensional face for being
a standard triangle or a unit parallelogram, and compute $\ell(F^\circ)$
as the lattice length of the edge of $\Delta^\circ$ dual to it.  The
ancillary input manifest pins these files to the revision and SHA-256
digests specified in \cite[\S6.1]{paper1}.  The full database-copy
verification reported there checks reflexivity and absence of repetitions
up to lattice equivalence for every entry; comparison with the published
classification total establishes completeness.  This verifies the input
copy, separately from the singularity enumeration reported here.
The verification report and input-provenance record accompany the
ancillary programs.  Our enumeration
finds exactly $198{,}849$ reflexive $4$-polytopes in the
Batyrev--Kreuzer class: every two-dimensional face a standard triangle
or a unit
parallelogram, and at least one parallelogram, so that $X$ is nodal with
at least one node (polytopes all of whose faces are standard triangles
give smooth $X$ and are excluded).  The total is
in perfect agreement with the count of Batyrev--Kreuzer \cite{BK10}, an
independent check on both computations.  Of these $198{,}849$, only
$3{,}774$ ($1.90\%$) have \emph{all} their square faces of dual length
$\ell(F^\circ)\ge 2$, and are thus smoothable by the per-face argument
alone; the remaining $195{,}075$ each carry at least one
$\ell(F^\circ)=1$ square.  The all-$\ell\ge2$ fraction falls sharply with
the vertex count: $39$ of the $43$ such polytopes with at most $7$
vertices are all-$\ell\ge2$, against $100$ of the $148{,}959$ with $14$
or more vertices.  Batyrev--Kreuzer show that $30{,}241$ of the $198{,}849$
are smoothable, via Namikawa's criterion.  Their criterion constrains only
the faces with $\ell(F^\circ) = 1$: in the form stated in \cite{BK10},
$X$ is smoothable if and only if for every square face with
$\ell(F^\circ) = 1$ the diagonal relation of that face lies in the span
of the diagonal relations of the other square faces; faces with
$\ell(F^\circ) \ge 2$ impose no condition, since their repeated class
cancels against itself.  A polytope whose squares all have
$\ell(F^\circ) \ge 2$ therefore passes the criterion vacuously, and the
$3{,}774$ all-$\ell\ge2$ polytopes all lie among the $30{,}241$: exactly
$3{,}774$ of Batyrev--Kreuzer's $30{,}241$ smoothings, one in eight, are
explained by the per-face mechanism, and each of the remaining
$26{,}467$ involves an $\ell(F^\circ)=1$ square whose exceptional class
occurs in a homology relation involving the classes over other
faces.  Understanding the
combinatorial structure of these cross-face relations (the analogue,
for the Calabi--Yau hypersurface, of the admissible Minkowski data of
\cite{CHP24}) is the central open problem left by this paper.

\section{Open questions}\label{sec:open}

\begin{question}\label{q:ell1}
Is a Batyrev Calabi--Yau threefold whose unique singular point has
germ $C(F)$, with $F$ of type \textup{(S)}, $i(F)\ge2$, and
$\ell(F^\circ)=1$, ever smoothable?
The case $i(F)=0$ is excluded by Theorem~\ref{thm:B2}; the case
$i(F)=1$ is excluded by Proposition~\ref{prop:C2} and the criterion
of \cite{paper5}.  For $i(F)\ge2$ the crepant resolution has a
reducible exceptional divisor, and these exact germs are not covered
by that criterion.
\end{question}

\begin{question}\label{q:multi}
Is Theorem~\ref{thm:Amulti}(ii) unconditional?  Can the
face families of several type-\textup{(S)} faces be realized
simultaneously (a codimension-two analogue of the admissible Minkowski
data of \cite{CHP24})?  Separately, is the semiuniversal base of a Batyrev
hypersurface with only type-\textup{(S)} faces always irreducible?  The
construction does not provide the compatible toric data needed for
iteration (Remark~\ref{rem:noiterate}).
\end{question}

\begin{question}
Can the nodal relation criterion be extended to all locally smoothable
cone singularities arising on unit-edge Batyrev hypersurfaces, using the
homology of the exceptional loci and the discriminants of the local
smoothing components?  The criterion of \cite{paper5} gives such an
extension for nodes and exact anticanonical cones over smooth del Pezzo
surfaces of degrees five, six and seven, and over $\PP^1\times\PP^1$.
The remaining toric germs include the type-\textup{(S)} polygons with
at least two interior lattice points.
\end{question}

\bigskip
\noindent{\sc Bernd Johannes Wuebben, New York, NY,} \texttt{wuebben@gmail.com}


\begin{thebibliography}{CHP24}
\bibitem[Alt95]{Altmann95} K.~Altmann, \emph{Minkowski sums and homogeneous
deformations of toric varieties}, T\^ohoku Math.\ J.\ (2) \textbf{47}
(1995), 151--184.
\bibitem[Alt97]{Altmann97} K.~Altmann, \emph{The versal deformation of an
isolated toric Gorenstein singularity}, Invent.\ Math.\ \textbf{128}
(1997), 443--479.
\bibitem[Bat94]{Batyrev94} V.~V. Batyrev, \emph{Dual polyhedra and mirror
symmetry for Calabi--Yau hypersurfaces in toric varieties}, J.\ Algebraic
Geom.\ \textbf{3} (1994), 493--535.
\bibitem[BK10]{BK10} V.~Batyrev and M.~Kreuzer, \emph{Constructing new
Calabi--Yau 3-folds and their mirrors via conifold transitions},
Adv.\ Theor.\ Math.\ Phys.\ \textbf{14} (2010), 879--898.
\bibitem[CFP22]{CFP} A.~Corti, M.~Filip, and A.~Petracci, \emph{Mirror symmetry
and smoothing Gorenstein toric affine 3-folds}, Facets of Algebraic
Geometry vol.~I, LMS Lecture Note Ser.\ \textbf{472}, CUP, 2022,
132--163.
\bibitem[CHP24]{CHP24} A.~Corti, P.~Hacking, and A.~Petracci, \emph{Smoothing
Gorenstein toric Fano 3-folds}, arXiv:2412.06500.
\bibitem[CLS11]{CLS} D.~A. Cox, J.~B. Little, and H.~K. Schenck, \emph{Toric
Varieties}, Graduate Studies in Mathematics \textbf{124}, Amer.\ Math.\
Soc., 2011.
\bibitem[DLRS10]{DLRS} J.~A. De~Loera, J.~Rambau, and F.~Santos,
\emph{Triangulations: Structures for Algorithms and Applications},
Algorithms and Computation in Mathematics \textbf{25}, Springer, 2010.
\bibitem[Fri86]{Friedman86} R.~Friedman, \emph{Simultaneous resolution of
threefold double points}, Math.\ Ann.\ \textbf{274} (1986), 671--689.
\bibitem[Gra74]{Grauert74} H.~Grauert, \emph{Der Satz von Kuranishi f\"ur
kompakte komplexe R\"aume}, Invent.\ Math.\ \textbf{25} (1974), 107--142.
\bibitem[Gro97]{Gross97} M.~Gross, \emph{Deforming Calabi--Yau threefolds},
Math.\ Ann.\ \textbf{308} (1997), 187--220.
\bibitem[Har77]{Hartshorne} R.~Hartshorne, \emph{Algebraic Geometry}, GTM
\textbf{52}, Springer, 1977.
\bibitem[Kap09]{Kapustka09} G.~Kapustka, \emph{Primitive contractions of
Calabi--Yau threefolds II}, J.\ Lond.\ Math.\ Soc.\ (2) \textbf{79}
(2009), 259--271.
\bibitem[Kaw08]{Kawamata08} Y.~Kawamata, \emph{Flops connect minimal
models}, Publ.\ Res.\ Inst.\ Math.\ Sci.\ \textbf{44} (2008), 419--423.
\bibitem[KK09]{KapustkaKapustka09} G.~Kapustka and M.~Kapustka,
\emph{Primitive contractions of Calabi--Yau threefolds I},
Comm.\ Algebra \textbf{37} (2009), 482--502.
\bibitem[Mat89]{Matsumura} H.~Matsumura, \emph{Commutative Ring Theory},
Cambridge Studies in Advanced Mathematics \textbf{8}, Cambridge
University Press, 1989.
\bibitem[Mav04]{Mav04} A.~R. Mavlyutov, \emph{Deformations of Calabi--Yau
hypersurfaces arising from deformations of toric varieties},
Invent.\ Math.\ \textbf{157} (2004), 621--633.
\bibitem[Nam02]{Namikawa02} Y.~Namikawa, \emph{Stratified local moduli of
Calabi--Yau threefolds}, Topology \textbf{41} (2002), 1219--1237.
\bibitem[Pet20]{Petracci20} A.~Petracci,
\emph{Some examples of non-smoothable Gorenstein Fano toric threefolds},
Math.\ Z.\ \textbf{295} (2020), 751--760; arXiv:1804.07960.
\bibitem[Pet21]{Petracci18} A.~Petracci, \emph{Homogeneous deformations of
toric pairs}, Manuscripta Math.\ \textbf{166} (2021), 37--72.
\bibitem[Rim80]{Rim} D.~S. Rim, \emph{Equivariant $G$-structure on versal
deformations}, Trans.\ Amer.\ Math.\ Soc.\ \textbf{257} (1980), 217--226.
\bibitem[RT09]{RollenskeThomas} S.~Rollenske and R.~Thomas, \emph{Smoothing
nodal Calabi--Yau $n$-folds}, J.\ Topol.\ \textbf{2} (2009), 405--421.
\bibitem[Wue26a]{paper1} B.~J. Wuebben, \emph{Non-smoothable Calabi--Yau
threefolds from reflexive polytopes}, preprint, 2026; available at
\url{https://github.com/bwuebben/calabi-yau-smoothability}.
\bibitem[Wue26e]{paper5} B.~J. Wuebben, \emph{Smoothing Calabi--Yau threefolds with nodes and del Pezzo cone points},
preprint, 2026, revised 5 September 2026.
\end{thebibliography}
\end{document}